\documentclass[preprint,12pt]{elsarticle}

\usepackage{amssymb}
\usepackage{amsmath}
\usepackage{graphicx}%
\usepackage{multirow}%
\usepackage{amsmath,amssymb,amsfonts}%
\usepackage{amsthm}%
\usepackage{mathrsfs}%
\usepackage[title]{appendix}%
\usepackage{xcolor}%
\usepackage{textcomp}%
\usepackage{manyfoot}%
\usepackage{booktabs}%
\usepackage{algorithm}%
\usepackage{algorithmicx}%
\usepackage{algpseudocode}%
\usepackage{listings}%
\usepackage{physics}
\usepackage{hyperref}
\usepackage{booktabs}
\usepackage{graphicx}
\usepackage[a4paper,margin=1in]{geometry}
\usepackage{lmodern}
\usepackage{microtype}

\numberwithin{equation}{section}

\theoremstyle{plain}
\newtheorem{theorem}{Theorem}[section]
\newtheorem{lemma}[theorem]{lemma}
  
\newtheorem{assumption}[theorem]{assumption}    
\newtheorem{corollary}[theorem]{corollary}

\theoremstyle{remark}
\newtheorem{remark}[theorem]{Remark}

\begin{document}

\begin{frontmatter}

\title{The Energy Based Near Singularity for Fourier Spectral 3D Navier-Stokes Equations}

\author{Beibei Li}

\begin{abstract}
We investigate the three-dimensional incompressible Navier-Stokes equations. The equations are discretized with Fourier spectral method and a fourth-order Runge-Kutta scheme in time. 
The spectral accuracy, resolution conditions, and an energy based conditional regularity framework are established analytically. Then we prove exponential convergence, algebraic convergence, and an a posteriori criterion that links numerical blowup to loss of regularity.
This work develops a suite of diagnostics for detecting potential finite time singular behavior.
\end{abstract}


\end{frontmatter}

\section{Introduction}
The three-dimensional incompressible Navier-Stokes equations are central to both fluid dynamics and the theory of partial differential equations. Since Leray established the existence of finite energy weak solutions and the energy inequality~\cite{Leray1934}, the global regularity in three dimensions has been clearly framed.

The classical milestones by Prodi, Serrin, and Ladyzhenskaya relate integrability conditions on the velocity field to smoothness and uniqueness of solutions~\cite{Prodi1959,Serrin1962,Ladyzhenskaya1967}. The works of Fujita-Kato and Kato developed local and global well posedness theories for small initial data in critical type settings and shaped much of the modern functional-analytic framework for Navier-Stokes~\cite{FujitaKato1964,Kato1984}. The well posedness was later obtained in~\cite{KochTataru2001}. The partial regularity theory identifies a potentially singular set of suitable weak solutions~\cite{CaffarelliKohnNirenberg1982}, while borderline criteria such as the $L_t^\infty L_x^{3}$ condition further constrain possible blowup scenarios~\cite{EscauriazaSereginSverak2003}.  The Beale-Kato-Majda criterion remains a key mechanism~\cite{BealeKatoMajda1984}. The axisymmetric configurations, including flows without swirl, provide additional settings in which global regularity or conditional regularity can be established and compared to the fully general 3D problem~\cite{UkhoviskiiYudovich1968}. The recent convex integration results demonstrate nonuniqueness for finite energy weak solutions~\cite{BuckmasterVicol2019}, highlighting limitations of energy-based selection and motivating finer structural understanding.

There is a comprehensive foundation for both the analytical and numerical study of the Navier-Stokes equations, including the works of Foias, Temam, and Leray’s original construction of weak solutions and energy inequalities~\cite{TemamNS,TemamNF,Leray1934,RobinsonNSReg,BerselliEnergyIneq}. The multiscale nature and its effective coarse grained description fit naturally within the general framework of multiscale modeling developed by E~\cite{e2011principles}. The refined regularity results such as Gevrey class and analyticity estimates further quantify the smoothing properties of Navier-Stokes dynamics~\cite{FoiasTemam1989,KukavicaVicol2002}, while a posteriori regularity criteria based on numerical computations connect rigorous analysis with high resolution simulations~\cite{RobinsonNSReg}. The spectral methods and their variants on the numerical side are systematically developed in the standard references~\cite{CanutoHussainiQuarteroniZang1988,GottliebOrszag1977,Trefethen2000}, with time discretization and stability issues analyzed in detail in, for example,~\cite{GottliebTurkel1980,HairerNorsettWanner1993}.

Tao’s averaged Navier-Stokes model provides an example of finite time blowup~\cite{Tao2016_JaMS_averagedNSBlowup}. The recent work establishes finite time singularities for the 3D Euler equations in a strong function class for blowup mechanisms~\cite{CordobaMartinezZoroaZheng2025_annPDE_EulerBlowup}. Hou investigates potentially singular behaviour for the 3D Navier-Stokes equations using high accuracy computations with rigorous error control~\cite{Hou2023_NavierStokes_PotentiallySingular_FoCM} and reports nearly self similar concentration in generalized axisymmetric Navier-Stokes and Boussinesq models~\cite{Hou2024_NearlySelfSimilar_gaxisymNS_Boussinesq_arXiv}.

We have a forcing for Navier-Stokes and conduct a computation guided study using Fourier spectral method with RK4 time stepping, tracking critical diagnostics, including the growth of the velocity, vorticity, a BKM type integral proxy, and energy dissipation consistency to distinguish genuine singular behavior from numerical artefacts and to probe potential blowup. 
In parallel, we develop an abstract framework for spectral accuracy, resolution conditions, and energy based conditional regularity, which we later specialize to Navier-Stokes, allowing us to interpret the  numerical breakdown within a controlled analytic setting.

\section{The Fourier Spectral for three dimensional Navier-Stokes equations}
The incompressible equations are
\begin{eqnarray}
&\frac{\partial \mathbf{u}}{\partial t} + \mathbf{u}\cdot\nabla\mathbf{u} 
= -\nabla p + \nu \nabla^2 \mathbf{u} + \mathbf{f}(\mathbf{x},t), \\
&\nabla \cdot \mathbf{u} = 0.
\end{eqnarray}

We consider the three dimensional incompressible equations
\begin{eqnarray}
\frac{\partial{u}}{\partial t} +  u \frac{\partial u}{\partial x}+v\frac{\partial u}{\partial y}+w\frac{\partial u}{\partial z}=\nu(\frac{\partial^2 u}{\partial x^2}+\frac{\partial^2 u}{\partial y^2}+\frac{\partial^2 u}{\partial z^2}) -\frac{\partial p}{\partial x} + f_x(x,y,z,t)\\
\frac{\partial{v}}{\partial t} +  u \frac{\partial v}{\partial x}+v\frac{\partial v}{\partial y}+w\frac{\partial v}{\partial z}=\nu(\frac{\partial^2 v}{\partial x^2}+\frac{\partial^2 v}{\partial y^2}+\frac{\partial^2 v}{\partial z^2}) -\frac{\partial p}{\partial y}+f_y(x,y,z,t)\\
\frac{\partial{w}}{\partial t} +  u \frac{\partial w}{\partial x}+v\frac{\partial w}{\partial y}+w\frac{\partial w}{\partial z}=\nu(\frac{\partial^2 w}{\partial x^2}+\frac{\partial^2 w}{\partial y^2}+\frac{\partial^2 w}{\partial z^2}) -\frac{\partial p}{\partial z}+f_z(x,y,z,t)
\end{eqnarray}
where $\mathbf{u}=(u,v,w)$ is the velocity field, $\mathbf{x}=(x,y,z)$, $p$ is the pressure, $\nu$ is the kinematic viscosity, and the incompressible is
\begin{eqnarray}
\frac{\partial u}{\partial x} + \frac{\partial v}{\partial y}+ \frac{\partial w}{\partial z}= 0
\end{eqnarray}

\subsubsection*{The Fourier Transform equations}
The Navier-Stokes equation becomes
\begin{eqnarray}
\partial_t \widehat{\mathbf{u}}
= \widehat{N} - \nu K^2\,\widehat{\mathbf{u}}+\hat{\mathbf{f}}(\mathbf{x},t)(\mathbf{k},t)
- i\,\mathbf{k}\,\widehat{p}(\mathbf{k},t).
\end{eqnarray}

In Fourier space, incompressibility reads
\begin{eqnarray}
i\,\mathbf{k}\cdot\widehat{\mathbf{u}} = 0.
\end{eqnarray}

\section{The Blowup Numerical Results}
We solve the Navier-Stokes equations in Fourier space using a fourth-order Runge-Kutta scheme with adaptive time step changed by CFL. The computation with resolution $258*258*258$ for a particularly concentrated initial condition and forcing goes up at $t=0.0002$. 
The maximum velocity grows to $\mathcal{O}(10^{5})$ and the kinetic energy to $\mathcal{O}(10^{7})$ in Figure~\ref{bkme}, while the vorticity maximum reaches $\mathcal{O}(10^{7})$ and the BKM integral exceeds $\mathcal{O}(10^{3})$ in Figure~\ref{Blowupsdiagnostics}, after which several fields go to infinity numerically. 
The numerical surrogate of the Beale-Kato-Majda necessary condition which we find that the accumulated integral of $\|\omega(\cdot,t)\|_\infty$ is already very large, strongly accelerating, and essentially monotone near the terminal time, which is consistent with divergence of the BKM integral. The energy balance and dissipation remain consistent within resolution limits. Taken together, these diagnostics indicate that for this choice of initial data and forcing, the numerical solution develops behavior consistent with a finite time singularity.

\begin{figure}[h]
\centering
\includegraphics[width=\linewidth]{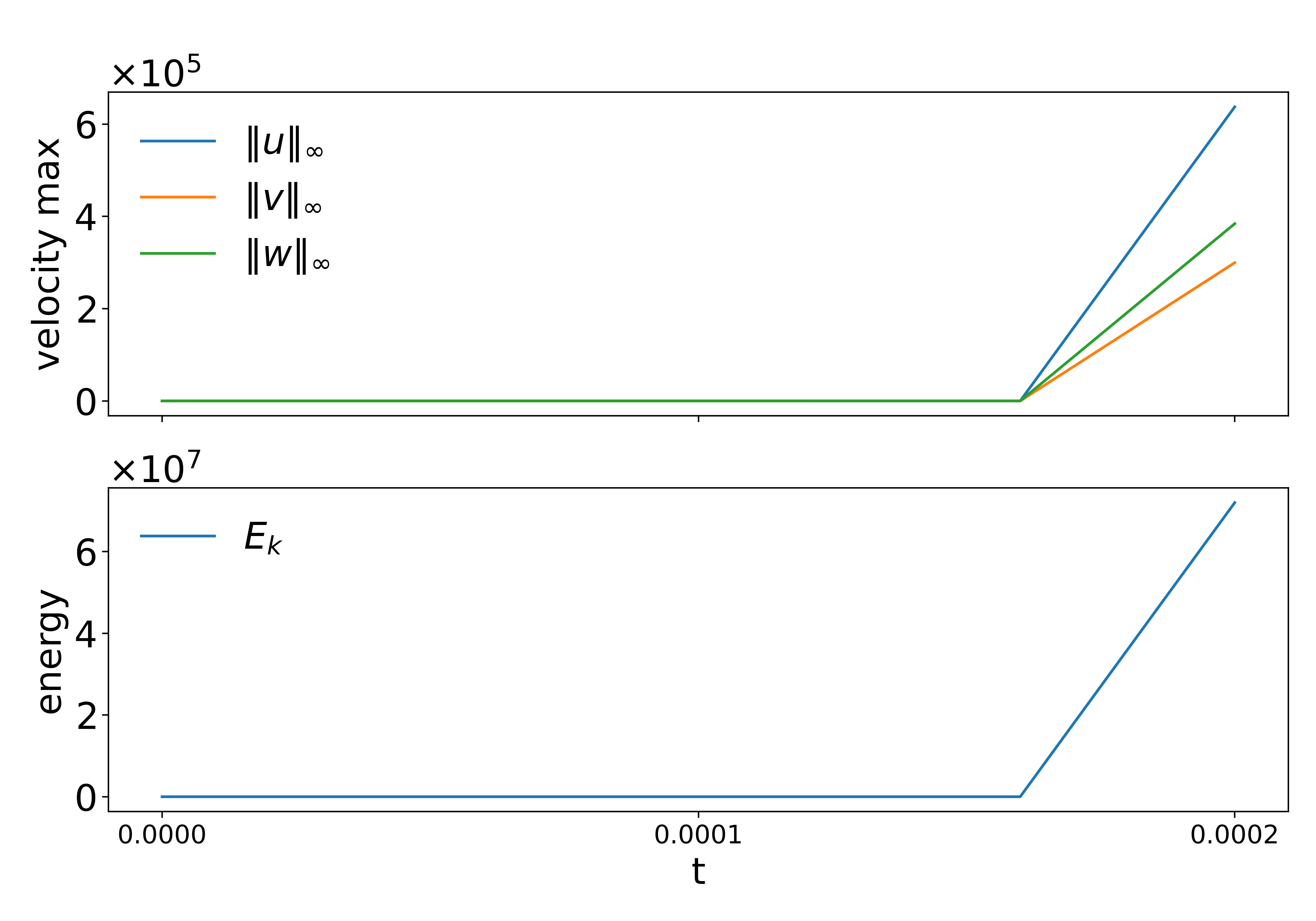}
\caption{The max velocity and energy}
\label{bkme}
\end{figure}

\begin{figure}[h]
\centering
\includegraphics[width=\linewidth]{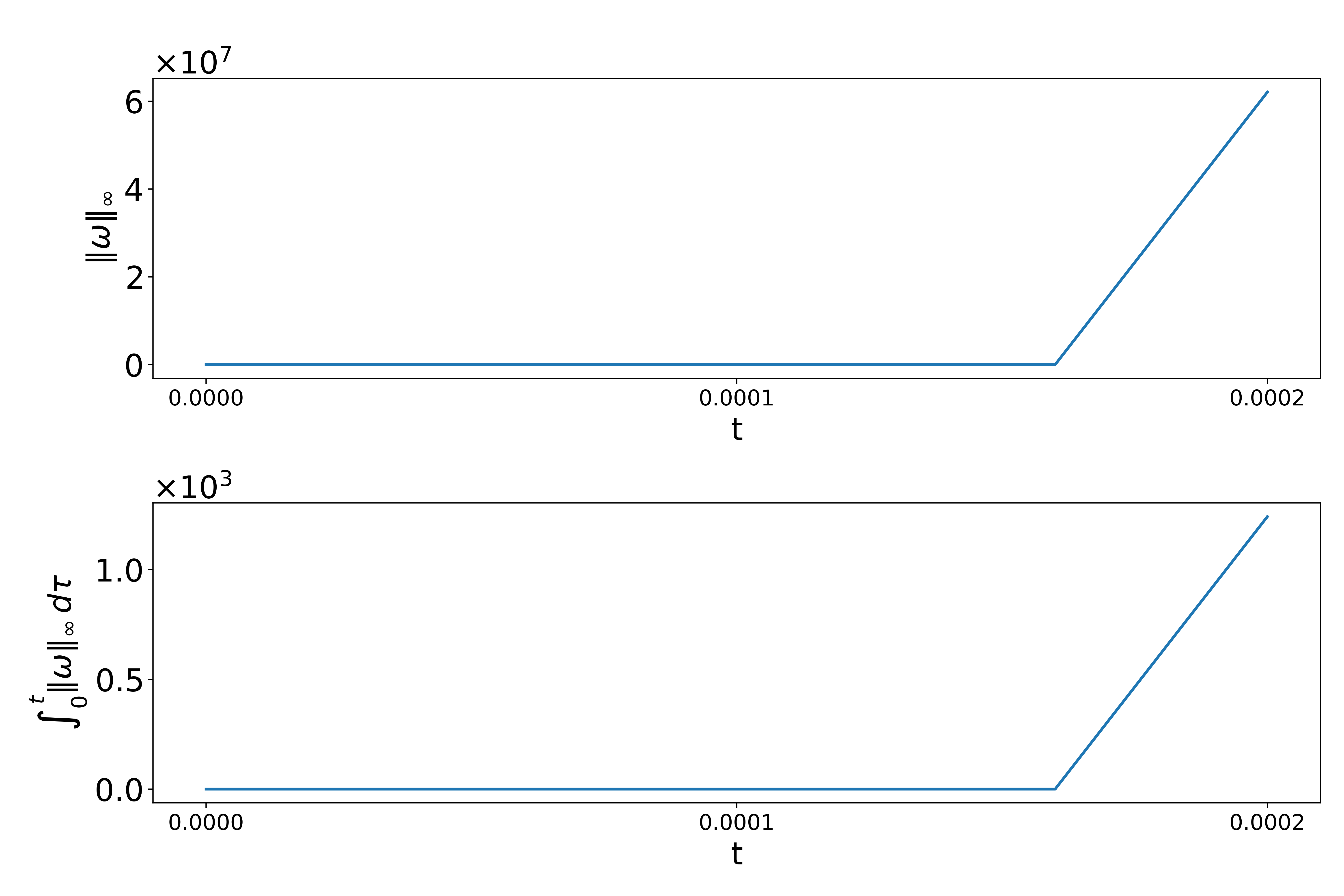}
\caption{The vorticity and BKM integral}
\label{Blowupsdiagnostics}
\end{figure}

\section{Spectral accuracy, resolution conditions, and energy based conditional regularity}
The goal in this section is two. We establish spectral convergence and algebraic convergence  for the fully discrete approximation. We formulate an energy-based conditional regularity framework that connects numerical breakdown with possible loss of regularity of the PDE.

We consider abstract evolution equation
\begin{equation}\label{eq:PDE}
  \partial_t u = \nu \Delta u + N(u),
  \qquad x\in\Omega:=[0,2\pi]^3,\quad t\in[0,T],
\end{equation}
where $u:\Omega\times[0,T]\to\mathbb{R}^m$ is a vector field,
$\nu>0$ is the viscosity, and $N(u)$ denotes a quadratic nonlinearity
e.g.\ the Navier-Stokes nonlinearity composed with the Leray projector.

We expand the $u$ in Fourier series
\[
  u(x,t) = \sum_{k\in\mathbb{Z}^3} \widehat{u}_k(t)\,e^{ik\cdot x},
  \qquad
  \widehat{u}_k(t)
  = \frac{1}{(2\pi)^3}\int_{[0,2\pi]^3} u(x,t)\,e^{-ik\cdot x}\,dx.
\]

\subsection{The analyticity strip and the exponential decay}
The implication is classical \cite{CanutoHussainiQuarteroniZang1988, GottliebOrszag1977}. In the Navier-Stokes the analyticity and regularity results showing exponential decay of the spectrum can be found in~\cite{FoiasTemam1989, KukavicaVicol2002}.
\begin{assumption}[analyticity strip]\label{ass:a}
The map $x\mapsto u(x,t)$ for each $t\in[0,T]$ admits a holomorphic 
extension to the complex strip
\[
  S_{\delta(t)} := \bigl\{
    z\in\mathbb{C}^3 : |\Im z_j|<\delta(t),\ j=1,2,3
  \bigr\},
\]
and there exists a function $M:[0,T]\to(0,\infty)$ such that
\[
  \|u(\cdot,t)\|_{L^\infty(S_{\delta(t)})} \le M(t)
  \quad\text{for all } t\in[0,T].
\]
There exists $\delta_{\min}>0$ such that
\[
  \delta(t)\ge\delta_{\min},\qquad \forall\,t\in[0,T].
\]
We set
\[
  C^*(T) := \sup_{0\le s\le T} \|u(\cdot,s)\|_{L^\infty(S_{\delta(s)})}.
\]
\end{assumption}

\begin{lemma}[analytic strip $\Rightarrow$ exponential decay]\label{lem:analytic-decay}
We then have for all $t\in[0,T]$ and all $k\in\mathbb{Z}^3$,
\begin{equation}\label{eq:3d-exp-decay}
  |\widehat{u}_k(t)| \le C^*(T)\,e^{-\delta_{\min}|k|},
  \qquad |k| := (k_1^2+k_2^2+k_3^2)^{1/2}.
\end{equation}
\end{lemma}

\begin{proof}
We fix $t\in[0,T]$ and choose $0<\rho<\delta(t)$.
Let $\eta\in\mathbb{R}^3$ be a vector with $|\eta_j|=\rho$ and such that
$k\cdot\eta \le -\rho|k|$; for instance,
$\eta_j = -\rho\,\mathrm{sgn}(k_j)$ with $\mathrm{sgn}(0):=0$ satisfies
this property. By analyticity and periodicity, we can shift the contour of
integration from $[0,2\pi]^3$ to $[0,2\pi]^3+i\eta$:
\[
  \widehat{u}_k(t)
  = \frac{1}{(2\pi)^3}
    \int_{[0,2\pi]^3} u(x+i\eta,t)\,e^{-ik\cdot(x+i\eta)}\,dx.
\]
Taking absolute values and using $|\int g|\le\int|g|$, we obtain
\[
  |\widehat{u}_k(t)|
  \le \frac{1}{(2\pi)^3}
    \int_{[0,2\pi]^3}
      |u(x+i\eta,t)|\,\bigl|e^{-ik\cdot(x+i\eta)}\bigr|\,dx.
\]
Since $x+i\eta\in S_\rho\subset S_{\delta(t)}$, we have
$|u(x+i\eta,t)|\le\|u(\cdot,t)\|_{L^\infty(S_\rho)}$.
\[
  e^{-ik\cdot(x+i\eta)}
  = e^{-ik\cdot x}\,e^{-ik\cdot i\eta}
  = e^{-ik\cdot x}\,e^{k\cdot\eta},
\]
so $|e^{-ik\cdot(x+i\eta)}| = e^{k\cdot\eta}\le e^{-\rho|k|}$ by the
choice of $\eta$. Hence
\[
  |\widehat{u}_k(t)|
  \le \frac{1}{(2\pi)^3}\int_{[0,2\pi]^3}
        \|u(\cdot,t)\|_{L^\infty(S_\rho)} e^{-\rho|k|}\,dx
  = \|u(\cdot,t)\|_{L^\infty(S_\rho)} e^{-\rho|k|}.
\]
This holds for every $0<\rho<\delta(t)$.
Letting $\rho\uparrow\delta(t)$ and using $\delta(t)\ge\delta_{\min}$
yields
\[
  |\widehat{u}_k(t)|
  \le \|u(\cdot,t)\|_{L^\infty(S_{\delta(t)})} e^{-\delta(t)|k|}
  \le C^*(T)\,e^{-\delta_{\min}|k|},
\]
which is \eqref{eq:3d-exp-decay}.
\end{proof}

\subsection{Spectral truncation error in $L^2$}
The truncation error for analytic functions decays exponentially fast in $K$; this is standard in the theory of spectral methods \cite{CanutoHussainiQuarteroniZang1988, GottliebOrszag1977}.

Let $P_K$ for $K\ge 1$ denote the Fourier projector onto modes with
$|k|\le K$,
\[
  (P_K u)(x,t) := \sum_{|k|\le K}\widehat{u}_k(t)e^{ik\cdot x},
\]
and define the truncation operator $R_K:=I-P_K$, so that
$R_K u = u-P_K u$.

\begin{lemma}[The spectral truncation error]\label{lem:spectral-trunc-3d}
There exists $C_1(T)>0$ such that for all $t\in[0,T]$ and all $K\ge1$, under the assumptions of lemma~\ref{lem:analytic-decay},
\begin{equation}\label{eq:spectral-trunc-3d}
  \|R_K u(\cdot,t)\|_{L^2(\Omega)}
  \le C_1(T)\,(1+K)e^{-\delta_{\min}K}.
\end{equation}
\end{lemma}

\begin{proof}
By Parseval's identity on $\Omega=[0,2\pi]^3$,
\[
  \|R_K u(\cdot,t)\|_{L^2(\Omega)}^2
  = \sum_{|k|>K} |\widehat{u}_k(t)|^2.
\]
We have
$
  |\widehat{u}_k(t)| \le C^*(T)e^{-\delta_{\min}|k|}
$
by using lemma~\ref{lem:analytic-decay},
and hence
\[
  \|R_K u(\cdot,t)\|_{L^2(\Omega)}^2
  \le C^*(T)^2 \sum_{|k|>K} e^{-2\delta_{\min}|k|}.
\]
We estimate the tail by grouping modes into shells. Define
\[
  N(m) := \#\{k\in\mathbb{Z}^3 : m<|k|\le m+1\},
  \qquad m=0,1,2,\dots.
\]
Since the number of lattice points in a shell of radius $m$ in three dimensions satisfies
\[
  N(m)\le C(1+m)^2,
\]
we have
\[
  \sum_{|k|>K} e^{-2\delta_{\min}|k|}
  = \sum_{m\ge K}
      \sum_{\substack{k\in\mathbb{Z}^3\\ m<|k|\le m+1}}
      e^{-2\delta_{\min}|k|}
  \le \sum_{m\ge K} N(m)e^{-2\delta_{\min}m}
  \le C\sum_{m\ge K}(1+m)^2 e^{-2\delta_{\min}m}.
\]
It remains to bound the last sum. Let $a:=2\delta_{\min}>0$. Then
\[
  \sum_{m\ge K}(1+m)^2 e^{-am}
  = e^{-aK}\sum_{n\ge0}(1+K+n)^2 e^{-an}.
\]
Since
\[
  1+K+n \le (1+K)(1+n),
\]
it follows that
\[
  (1+K+n)^2 \le (1+K)^2(1+n)^2.
\]
Therefore,
\[
  \sum_{m\ge K}(1+m)^2 e^{-am}
  \le (1+K)^2 e^{-aK}\sum_{n\ge0}(1+n)^2 e^{-an}.
\]
The series
\[
  \sum_{n\ge0}(1+n)^2 e^{-an}
\]
converges, since exponential decay dominates polynomial growth. Hence there exists a constant $C_a>0$ such that
\[
  \sum_{m\ge K}(1+m)^2 e^{-am}
  \le C_a(1+K)^2 e^{-aK}.
\]
Substituting back $a=2\delta_{\min}$ gives
\[
  \sum_{|k|>K} e^{-2\delta_{\min}|k|}
  \le C'(1+K)^2 e^{-2\delta_{\min}K}.
\]

Therefore,
\[
  \|R_K u(\cdot,t)\|_{L^2(\Omega)}^2
  \le C^*(T)^2 C'(1+K)^2 e^{-2\delta_{\min}K}.
\]
Taking square roots yields
\[
  \|R_K u(\cdot,t)\|_{L^2(\Omega)}
  \le C_1(T)(1+K)e^{-\delta_{\min}K},
\]
for a suitable constant $C_1(T)>0$.
\end{proof}

\subsection{$L^\infty$ bounds and spectral convergence}
We can also control the exact solution
and its spectral truncations in the $L^\infty$ norm under assumption~\ref{ass:a}.

\begin{lemma}[$L^\infty$ bound for exact solution]\label{lem:uinfty-bound}
There exists a constant
$C_\infty(T)>0$ under assumption~\ref{ass:a} for every $T>0$ such that
\[
  \sup_{0\le t\le T} \|u(\cdot,t)\|_{L^\infty(\Omega)} \le C_\infty(T).
\]
\end{lemma}

\begin{proof}
By lemma~\ref{lem:analytic-decay}, for all $t\in[0,T]$ and
$k\in\mathbb{Z}^3$,
\(
  |\widehat{u}_k(t)| \le C^*(T) e^{-\delta_{\min}|k|}.
\)
Hence
\[
  \|u(\cdot,t)\|_{L^\infty}
  \le \sum_{k\in\mathbb{Z}^3} |\widehat{u}_k(t)|
  \le C^*(T) \sum_{k\in\mathbb{Z}^3} e^{-\delta_{\min}|k|}
  =: C_\infty(T),
\]
and the series on the right converges since it is dominated by a
geometric series. This bound is uniform in $t\in[0,T]$.
\end{proof}

We now obtain an exponential $L^\infty$ spectral convergence estimate.

\begin{lemma}[$L^\infty$ spectral truncation error]\label{lem:spectral-trunc-Linf}
Let $P_K$ be the Fourier projector defined above and let $R_K=I-P_K$.
There exist constants $C_\infty(T)>0$ and $c\in(0,\delta_{\min})$ such that
for all $t\in[0,T]$ and $K\ge1$, under assumption~\ref{ass:a},
\[
  \|R_Ku(\cdot,t)\|_{L^\infty(\Omega)}
  \le C_\infty(T)e^{-cK}.
\]
\end{lemma}

\begin{proof}
We write
\[
  R_Ku(x,t)=\sum_{|k|>K}\widehat{u}_k(t)e^{ik\cdot x}.
\]
By lemma~\ref{lem:analytic-decay},
\[
  |\widehat{u}_k(t)|\le C^*(T)e^{-\delta_{\min}|k|},
\]
hence
\[
  \|R_Ku(\cdot,t)\|_{L^\infty}
  \le \sum_{|k|>K}|\widehat{u}_k(t)|
  \le C^*(T)\sum_{|k|>K}e^{-\delta_{\min}|k|}.
\]
We have by grouping Fourier modes into shells as in the proof of lemma~\ref{lem:spectral-trunc-3d}
\[
  \sum_{|k|>K}e^{-\delta_{\min}|k|}
  \le C\sum_{m\ge K}(1+m)^2e^{-\delta_{\min}m}
  \le C'(1+K)^2e^{-\delta_{\min}K}.
\]
Therefore,
\[
  \|R_Ku(\cdot,t)\|_{L^\infty}
  \le C(T)(1+K)^2e^{-\delta_{\min}K}.
\]
Since for any $c\in(0,\delta_{\min})$ the polynomial factor $(1+K)^2$ can be absorbed into the weaker exponential $e^{-cK}$, there exists a constant $C_\infty(T)>0$ such that
\[
  \|R_Ku(\cdot,t)\|_{L^\infty(\Omega)}
  \le C_\infty(T)e^{-cK}.
\]
\end{proof}

\subsection{The spectral spatial error and temporal error}
The convergence of spectral approximations including exponential convergence in the analytic
setting, is standard~\cite{CanutoHussainiQuarteroniZang1988, GottliebOrszag1977}.
This is a result from the ordinary differential equations numerical analysis~\cite{HairerNorsettWanner1993}.
Time discretization and stability issues for spectral methods are discussed
in the context of PDEs~\cite{GottliebTurkel1980}.

We now introduce the finite dimensional spectral approximation
$u_K$.

\begin{lemma}[The spatial error]\label{error}
Let
\[
B(a,b):=P\bigl((a\cdot\nabla)b\bigr),
\]
where $P$ is the Leray projector. We assume assumption~\ref{ass:a}. Let
$\mathcal F$ be a forcing operator such that, for every $T>0$, there exists a
neighborhood of the trajectories
\[
\{u(t):0\le t\le T\}\cup \{u_K(t):0\le t\le T,\ K\ge 1\}
\]
and a constant $L_T>0$ such that
\begin{equation}\label{eq:F-Hminus1}
\mathcal F(v)\in H^{-1}(\Omega)
,
\end{equation}
and
\begin{equation}\label{eq:F-Lip-L2-Hminus1}
\|\mathcal F(v)-\mathcal F(w)\|_{H^{-1}(\Omega)}
\le
L_T\|v-w\|_{L^2(\Omega)}
.
\end{equation}

Let $u_K$ be the spectral approximation solving
\[
\partial_t u_K
=
P_K\Bigl(
\nu\Delta u_K
-
P\bigl((u_K\cdot\nabla)u_K\bigr)
+
\mathcal F(u_K)
\Bigr),
\qquad
u_K(0)=P_Ku_0,
\]
and let the exact solution $u$ solve
\[
\partial_t u
=
\nu\Delta u
-
P\bigl((u\cdot\nabla)u\bigr)
+
\mathcal F(u).
\]
Then there exists a constant $C_2(T)>0$ such that for all $t\in[0,T]$ and
all $K\ge 1$,
\[
\|u_K(\cdot,t)-u(\cdot,t)\|_{L^2(\Omega)}
\le
C_2(T)(1+K)^2e^{-\delta_{\min}K}.
\]
Equivalently, after weakening the exponential rate, one may write
\[
\|u_K(\cdot,t)-u(\cdot,t)\|_{L^2(\Omega)}
\le
C_2(T)e^{-cK}
\qquad\text{for any }c\in(0,\delta_{\min}).
\]
\end{lemma}

\begin{proof}
We define
\[
e_K(t):=u_K(t)-P_Ku(t),
\qquad
r_K(t):=R_Ku(t)=u(t)-P_Ku(t).
\]
Then
\[
u_K-u=e_K-r_K,
\]
and therefore
\[
\|u_K-u\|_{L^2}\le \|e_K\|_{L^2}+\|r_K\|_{L^2}.
\]

Since $P_K$ commutes with $\Delta$, and since
\[
\partial_t(P_Ku)
=
\nu\Delta P_Ku-P_KB(u,u)+P_K\mathcal F(u),
\]
subtracting the equations for $u_K$ and $P_Ku$ yields
\[
\partial_t e_K
=
\nu\Delta e_K
-
P_K\bigl(B(u_K,u_K)-B(u,u)\bigr)
+
P_K\bigl(\mathcal F(u_K)-\mathcal F(u)\bigr).
\]
Because $u_K\in\operatorname{Ran}(P_K)$ and $P_Ku\in\operatorname{Ran}(P_K)$,
we have $e_K\in\operatorname{Ran}(P_K)$. Taking the $L^2$ inner product with
$e_K$ and using that $P_K$ is the orthogonal projector on $L^2$, we obtain
\begin{equation}\label{eq:ek-energy-general}
\frac12\frac{d}{dt}\|e_K\|_{L^2}^2
+\nu\|\nabla e_K\|_{L^2}^2
=
I_{\rm NS}+I,
\end{equation}
where
\[
I_{\rm NS}:=
-\bigl\langle B(u_K,u_K)-B(u,u),e_K\bigr\rangle,
\qquad
I:=
\bigl\langle \mathcal F(u_K)-\mathcal F(u),e_K\bigr\rangle_{H^{-1},H^1}.
\]

We first estimate the Navier--Stokes term. Then using
\[
u_K=P_Ku+e_K,
\qquad
u=P_Ku+r_K,
\]
we expand
\begin{eqnarray*}
B(u_K,u_K)-B(u,u)
&=
B(P_Ku+e_K,P_Ku+e_K)-B(P_Ku+r_K,P_Ku+r_K)\\
&=
B(e_K,P_Ku)+B(P_Ku,e_K)+B(e_K,e_K)\\
&\quad
-B(P_Ku,r_K)-B(r_K,P_Ku)-B(r_K,r_K).
\end{eqnarray*}
Introduce the standard trilinear form
\[
b(a,b,c):=\int_\Omega (a\cdot\nabla b)\cdot c\,dx.
\]
Then we have 
\[
\langle B(a,b),c\rangle=b(a,b,c),
\qquad
b(a,c,c)=0,
\qquad
b(a,b,c)=-b(a,c,b)
\]
for divergence-free vector fields.

Since $u$, $P_Ku$, $r_K$, $u_K$, and $e_K$ are all divergence free, we have
\[
\langle B(P_Ku,e_K),e_K\rangle=0,
\qquad
\langle B(e_K,e_K),e_K\rangle=0.
\]
Hence
\begin{eqnarray*}
I_{\rm NS}
&=
-\langle B(e_K,P_Ku),e_K\rangle
+\langle B(P_Ku,r_K),e_K\rangle\\
&\quad
+\langle B(r_K,P_Ku),e_K\rangle
+\langle B(r_K,r_K),e_K\rangle.
\end{eqnarray*}
Estimating term by term gives
\[
|\langle B(e_K,P_Ku),e_K\rangle|
\le
\|\nabla P_Ku\|_{L^\infty}\|e_K\|_{L^2}^2,
\]
\[
|\langle B(P_Ku,r_K),e_K\rangle|
=
|b(P_Ku,e_K,r_K)|
\le
\|P_Ku\|_{L^\infty}\|\nabla e_K\|_{L^2}\|r_K\|_{L^2},
\]
\[
|\langle B(r_K,P_Ku),e_K\rangle|
\le
\|\nabla P_Ku\|_{L^\infty}\|r_K\|_{L^2}\|e_K\|_{L^2},
\]
and
\[
|\langle B(r_K,r_K),e_K\rangle|
=
|b(r_K,e_K,r_K)|
\le
\|r_K\|_{L^\infty}\|\nabla e_K\|_{L^2}\|r_K\|_{L^2}.
\]

By the analyticity strip assumption, the Fourier coefficients of $u$ satisfy
\[
|\widehat{u}_k(t)|\le C(T)e^{-\delta_{\min}|k|},
\qquad k\in\mathbb Z^3,\quad t\in[0,T].
\]
Hence, uniformly in $K$,
\[
\|P_Ku(\cdot,t)\|_{L^\infty}\le C(T),
\qquad
\|\nabla P_Ku(\cdot,t)\|_{L^\infty}\le C(T).
\]
We infer
\[
I_{\rm NS}
\le
\frac{\nu}{4}\|\nabla e_K\|_{L^2}^2
+
C(T)\|e_K\|_{L^2}^2
+
C(T)\Bigl(
\|r_K\|_{L^2}^2+\|r_K\|_{L^\infty}^2\|r_K\|_{L^2}^2
\Bigr)
\]
using Young's inequality.
We now estimate the forcing term. By \eqref{eq:F-Lip-L2-Hminus1},
\[
\|\mathcal F(u_K)-\mathcal F(u)\|_{H^{-1}}
\le
L_T\|u_K-u\|_{L^2}
\le
L_T\bigl(\|e_K\|_{L^2}+\|r_K\|_{L^2}\bigr).
\]
Therefore,
\begin{eqnarray*}
|I|
&\le
\|\mathcal F(u_K)-\mathcal F(u)\|_{H^{-1}}\|e_K\|_{H^1}\\
&\le
L_T\bigl(\|e_K\|_{L^2}+\|r_K\|_{L^2}\bigr)
\bigl(\|e_K\|_{L^2}+\|\nabla e_K\|_{L^2}\bigr).
\end{eqnarray*}
By Young's inequality,
\[
|I|
\le
\frac{\nu}{4}\|\nabla e_K\|_{L^2}^2
+
C(T)\|e_K\|_{L^2}^2
+
C(T)\|r_K\|_{L^2}^2.
\]

Substituting the estimates for $I_{\rm NS}$ and $I$ into
\eqref{eq:ek-energy-general}, we obtain
\[
\frac{d}{dt}\|e_K\|_{L^2}^2
\le
C(T)\|e_K\|_{L^2}^2
+
C(T)\Bigl(
\|r_K\|_{L^2}^2+\|r_K\|_{L^\infty}^2\|r_K\|_{L^2}^2
\Bigr).
\]

By Lemma~\ref{lem:spectral-trunc-3d},
\[
\|r_K(\cdot,t)\|_{L^2}
\le
C_1(T)(1+K)e^{-\delta_{\min}K},
\]
and by the $L^\infty$ spectral truncation estimate,
\[
\|r_K(\cdot,t)\|_{L^\infty}
\le
C_\infty(T)e^{-cK}
\qquad\text{for some }c\in(0,\delta_{\min}).
\]
Hence
\[
\|r_K\|_{L^2}^2+\|r_K\|_{L^\infty}^2\|r_K\|_{L^2}^2
\le
C(T)(1+K)^2e^{-2\delta_{\min}K}.
\]
Therefore,
\[
\frac{d}{dt}\|e_K\|_{L^2}^2
\le
C(T)\|e_K\|_{L^2}^2
+
C(T)(1+K)^2e^{-2\delta_{\min}K}.
\]

Since
\[
e_K(0)=u_K(0)-P_Ku(0)=P_Ku_0-P_Ku_0=0,
\]
Gronwall's inequality yields
\[
\|e_K(t)\|_{L^2}^2
\le
C(T)(1+K)^2e^{-2\delta_{\min}K}.
\]
Thus
\[
\|e_K(t)\|_{L^2}
\le
C(T)(1+K)e^{-\delta_{\min}K}.
\]
Then
\[
\|u_K(t)-u(t)\|_{L^2}
\le
\|e_K(t)\|_{L^2}+\|r_K(t)\|_{L^2}
\le
C_2(T)(1+K)e^{-\delta_{\min}K}.
\]
In particular, after enlarging the constant, this implies
\[
\|u_K(t)-u(t)\|_{L^2}
\le
C_2(T)(1+K)^2e^{-\delta_{\min}K}.
\]
This completes the proof.
\end{proof}

\begin{lemma}[Temporal discretization error]\label{lem:time-error}
Let $u_K(t)$ solve the finite-dimensional system
\[
  \partial_t u_K = \mathrm{rhs}_K(u_K),\qquad u_K(0)=u_{K,0},
\]
where $\mathrm{rhs}_K$ is locally Lipschitz on a bounded set containing the
trajectory $u_K([0,T])$.
Let $u_{K,\Delta t}^n$ be defined by a one-step time discretization
\[
  u_{K,\Delta t}^{n+1} = \Phi_{\Delta t}(u_{K,\Delta t}^n),\qquad
  u_{K,\Delta t}^0 = u_K(0),
\]
of order $p\ge1$.
Then there exists $C_3(T)>0$ such that for all $t_n=n\Delta t\le T$,
\[
  \|u_{K,\Delta t}^n - u_K(t_n)\|_{L^2}
  \le C_3(T)\,(\Delta t)^p.
\]
\end{lemma}

\begin{proof}
We define
\[
  e_n := u_{K,\Delta t}^n - u_K(t_n).
\]
This gives by adding and subtracting $\Phi_{\Delta t}(u_K(t_n))$
\[
  e_{n+1}
  = \Phi_{\Delta t}(u_{K,\Delta t}^n) - \Phi_{\Delta t}(u_K(t_n))
    + \bigl(\Phi_{\Delta t}(u_K(t_n)) - u_K(t_{n+1})\bigr).
\]
By the Lipschitz stability of the one-step map on the relevant bounded set,
\[
  \|\Phi_{\Delta t}(v)-\Phi_{\Delta t}(w)\|_{L^2}
  \le (1+C\Delta t)\|v-w\|_{L^2}.
\]
Since the method is of order $p$, the local truncation error satisfies
\[
  \|\Phi_{\Delta t}(u_K(t_n)) - u_K(t_{n+1})\|_{L^2}
  \le C(\Delta t)^{p+1}.
\]
Hence
\[
  \|e_{n+1}\|_{L^2}
  \le (1+C\Delta t)\|e_n\|_{L^2} + C(\Delta t)^{p+1}.
\]
Because $e_0=0$, a discrete Gronwall argument yields
\[
  \|e_n\|_{L^2}\le C_3(T)(\Delta t)^p,\qquad t_n\le T.
\]
\end{proof}

\subsection{The fully discrete error and resolution condition}
The bounds of form combining exponential spectral accuracy in space with algebraic accuracy of order $p$ in time, are standard in the analysis of fully discrete spectral schemes~\cite{CanutoHussainiQuarteroniZang1988, GottliebOrszag1977, GottliebTurkel1980}.

\begin{theorem}[The fully discrete spectral accuracy]\label{thm:full-error-3d}
We assume \eqref{eq:PDE}, assumptions~\ref{ass:a},
and lemmas~\ref{lem:analytic-decay}, \ref{lem:spectral-trunc-3d},
\ref{error}, and~\ref{lem:time-error}.
Then there exist constants $C_1(T),C_2(T)>0$ such that for all
$t_n=n\Delta t\le T$,
\begin{equation}\label{eq:full-error-3d}
  \|u_{K,\Delta t}^n-u(\cdot,t_n)\|_{L^2(\Omega)}
  \le C_1(T)(1+K)^2e^{-\delta_{\min}K}+C_2(T)(\Delta t)^p.
\end{equation}
\end{theorem}

\begin{proof}
We decompose
\[
  u_{K,\Delta t}^n-u(\cdot,t_n)
  =\bigl(u_{K,\Delta t}^n-u_K(t_n)\bigr)
   +\bigl(u_K(t_n)-u(\cdot,t_n)\bigr).
\]
Taking $L^2(\Omega)$ norms and using the triangle inequality, we obtain
\[
  \|u_{K,\Delta t}^n-u(\cdot,t_n)\|_{L^2}
  \le \|u_{K,\Delta t}^n-u_K(t_n)\|_{L^2}
   + \|u_K(t_n)-u(\cdot,t_n)\|_{L^2}.
\]
The first term is bounded by lemma~\ref{lem:time-error}:
\[
  \|u_{K,\Delta t}^n-u_K(t_n)\|_{L^2}
  \le C_3(T)(\Delta t)^p.
\]
The second term is bounded by lemma~\ref{error}:
\[
  \|u_K(t_n)-u(\cdot,t_n)\|_{L^2}
  \le C_4(T)(1+K)^2e^{-\delta_{\min}K}.
\]
Renaming constants yields \eqref{eq:full-error-3d}.
\end{proof}

\begin{corollary}[Resolution condition]\label{cor:resolution}
Let $\varepsilon>0$ be a prescribed tolerance.
If $K$ and $\Delta t$ are chosen such that
\[
  C_1(T)(1+K)^2e^{-\delta_{\min}K}\le \frac{\varepsilon}{2},
  \qquad
  C_2(T)(\Delta t)^p\le \frac{\varepsilon}{2},
\]
then for all $t_n\in[0,T]$,
\[
  \|u_{K,\Delta t}^n-u(\cdot,t_n)\|_{L^2(\Omega)}\le \varepsilon.
\]
In particular, up to multiplicative constants, it suffices to take
\[
  K \gtrsim \frac{1}{\delta_{\min}}\log\frac{C}{\varepsilon},
  \qquad
  \Delta t \lesssim \left(\frac{\varepsilon}{C'}\right)^{1/p}.
\]
\end{corollary}

\begin{proof}
We have from Theorem~\ref{thm:full-error-3d},
\[
  \|u_{K,\Delta t}^n-u(\cdot,t_n)\|_{L^2}
  \le C_1(T)(1+K)^2e^{-\delta_{\min}K}+C_2(T)(\Delta t)^p.
\]
If each term on the right-hand side is at most $\varepsilon/2$, then the sum is
at most $\varepsilon$.

The second condition gives immediately
\[
  \Delta t \le \left(\frac{\varepsilon}{2C_2(T)}\right)^{1/p}.
\]

The factor $(1+K)^2$ is only polynomial in $K$, whereas
$e^{-\delta_{\min}K}$ decays exponentially. Hence, up to constants, the condition
\[
  C_1(T)(1+K)^2e^{-\delta_{\min}K}\le \frac{\varepsilon}{2}
\]
is ensured by taking
\[
  K \gtrsim \frac{1}{\delta_{\min}}\log\frac{C}{\varepsilon},
\]
absorbing numerical constants into $C$ and $C'$ yields the stated form.
\end{proof}

\subsection{The fully discrete $L^\infty$ error and convergence of diagnostics}

We obtain a fully discrete error bound.

\begin{lemma}[Semidiscrete $H^s$ spectral error]\label{lem:Hs-semidiscrete}
Let $s>\frac52$. We assume assumption~\ref{ass:a}. Let $u$ be the exact solution of
\eqref{eq:PDE} on $[0,T]$, and let $u_K$ be the semidiscrete Fourier
approximation. We assume further that there exists $M_T>0$, independent of $K$,
such that
\[
\sup_{0\le t\le T}\Bigl(\|u(\cdot,t)\|_{H^{s+1}(\Omega)}
+\|u_K(\cdot,t)\|_{H^{s+1}(\Omega)}\Bigr)\le M_T .
\]

Let $\mathcal F$ be a forcing operator such that
\[
\mathcal F:H^{s+1}(\Omega)\to H^{s-1}(\Omega),
\]
and assume that $\mathcal F$ is locally Lipschitz on bounded subsets of
$H^{s+1}(\Omega)$ in the sense that there exists a constant $L_T>0$ such that
\[
\|\mathcal F(v)-\mathcal F(w)\|_{H^{s-1}(\Omega)}
\le
L_T\|v-w\|_{H^s(\Omega)}
\]
for all $v,w$ with
\[
\|v\|_{H^{s+1}(\Omega)}+\|w\|_{H^{s+1}(\Omega)}\le 2M_T.
\]

Then for every $0<c<\delta_{\min}$ there exists a constant $C_s(T)>0$ such that
\[
\sup_{0\le t\le T}\|u_K(\cdot,t)-u(\cdot,t)\|_{H^s(\Omega)}
\le C_s(T)e^{-cK}.
\]
Consequently, by Sobolev embedding $H^s(\Omega)\hookrightarrow L^\infty(\Omega)$,
\[
\sup_{0\le t\le T}\|u_K(\cdot,t)-u(\cdot,t)\|_{L^\infty(\Omega)}
\le C_s(T)e^{-cK}.
\]
\end{lemma}

\begin{proof}
Set
\[
e_K:=u_K-P_Ku,
\qquad
r_K:=u-P_Ku.
\]
Then
\[
u_K-u=e_K-r_K.
\]

Write
\[
B(a,b):=P\bigl((a\cdot\nabla)b\bigr).
\]
Since $P_K$ commutes with $\Delta$ and $P$, subtracting the projected exact
equation from the semidiscrete one yields
\[
\partial_t e_K
=
\nu\Delta e_K
-
P_K\bigl(B(u_K,u_K)-B(u,u)\bigr)
+
P_KP\bigl(\mathcal F(u_K)-\mathcal F(u)\bigr).
\]

We apply $\Lambda^s=(I-\Delta)^{s/2}$ and take the $L^2$ inner product with
$\Lambda^s e_K$. By self-adjointness of $\Lambda$,
\[
\bigl|\langle \Lambda^s G,\Lambda^s e_K\rangle\bigr|
=
\bigl|\langle \Lambda^{s-1}G,\Lambda^{s+1}e_K\rangle\bigr|
\le \|G\|_{H^{s-1}}\|e_K\|_{H^{s+1}},
\]
and hence by Young's inequality,
\[
\bigl|\langle \Lambda^s G,\Lambda^s e_K\rangle\bigr|
\le \frac{\nu}{4}\|e_K\|_{H^{s+1}}^2 + C_\nu\|G\|_{H^{s-1}}^2 .
\]

The standard product estimates for $s>\frac52$ give
\[
\|B(a,b)\|_{H^{s-1}}
\le C_s \|a\|_{H^s}\|b\|_{H^s}.
\]
Moreover, by the local Lipschitz assumption on $\mathcal F$ and the uniform
$H^{s+1}$ bound,
\[
\|\mathcal F(u_K)-\mathcal F(u)\|_{H^{s-1}}
\le C_{T}\|u_K-u\|_{H^s}.
\]
Since
\[
u_K-u=e_K-r_K,
\]
it follows that
\[
\|B(u_K,u_K)-B(u,u)\|_{H^{s-1}}
+
\|\mathcal F(u_K)-\mathcal F(u)\|_{H^{s-1}}
\le C(T)\bigl(\|e_K\|_{H^s}+\|r_K\|_{H^s}\bigr).
\]
Therefore,
\[
\frac12\frac{d}{dt}\|e_K\|_{H^s}^2
+\nu\|e_K\|_{H^{s+1}}^2
\le
C(T)\|e_K\|_{H^s}^2
+
C(T)\bigl(\|r_K\|_{H^s}^2+\|r_K\|_{H^{s+1}}^2\bigr).
\]

By analyticity of $u$, the Fourier tail decays exponentially in every Sobolev
norm: for each integer $m\ge0$,
\[
\|r_K(\cdot,t)\|_{H^m(\Omega)}
\le C_m(T)(1+K)^m e^{-\delta_{\min}K}.
\]
Hence, for any $0<c<\delta_{\min}$,
\[
\|r_K(\cdot,t)\|_{H^s}+\|r_K(\cdot,t)\|_{H^{s+1}}
\le C_s(T)e^{-cK},
\qquad 0\le t\le T.
\]

Since $e_K(0)=0$, Gronwall's inequality yields
\[
\sup_{0\le t\le T}\|e_K(\cdot,t)\|_{H^s}
\le C_s(T)e^{-cK}.
\]
Then
\[
\|u_K(\cdot,t)-u(\cdot,t)\|_{H^s}
\le \|e_K(\cdot,t)\|_{H^s}+\|r_K(\cdot,t)\|_{H^s},
\]
so
\[
\sup_{0\le t\le T}\|u_K(\cdot,t)-u(\cdot,t)\|_{H^s}
\le C_s(T)e^{-cK}.
\]
The $L^\infty$ estimate follows from the embedding
$H^s(\Omega)\hookrightarrow L^\infty(\Omega)$ for $s>\frac32$.
\end{proof}

\begin{theorem}[The fully discrete $H^s$ error]\label{thm:full-error-Linf}
%\label{thm:fully_discrete_Hs}
Let $s > \frac52$. We assume that the exact solution $u$ satisfies the
regularity and analyticity hypotheses stated above, and that the
semi-discrete spectral solution $u_K$ satisfies
\begin{equation}
\sup_{0 \le t \le T} \|u_K(t)-u(t)\|_{H^s}
\le C_s(T)e^{-cK}
\label{eq:semidiscrete_Hs_error}
\end{equation}
by lemma \ref{lem:Hs-semidiscrete}.
We assume further that for each fixed $K$, the explicit Runge-Kutta method
of order $p$ with one-step map $\Phi_{\Delta t}^{(K)}$ satisfies
\begin{equation}
\|\Phi_{\Delta t}^{(K)}(v)-\Phi_{\Delta t}^{(K)}(w)\|_{H^s}
\le (1+L_K\Delta t)\|v-w\|_{H^s},
\label{eq:RK_stability_Hs}
\end{equation}
for all $v,w$ in a bounded $H^s$-neighborhood of the semi-discrete trajectory,
and
\begin{equation}
\|\Phi_{\Delta t}^{(K)}(u_K(t_n)) - u_K(t_{n+1})\|_{H^s}
\le J_K \Delta t^{p+1},
\qquad t_n=n\Delta t \le T.
\label{eq:RK_local_error_Hs}
\end{equation}
Define the fully discrete approximation by
\begin{equation}
u_{K,\Delta t}^{n+1} = \Phi_{\Delta t}^{(K)}(u_{K,\Delta t}^n),
\qquad
u_{K,\Delta t}^0 = P_K u_0.
\label{eq:fully_discrete_scheme}
\end{equation}
Then for every $t_n=n\Delta t \le T$,
\begin{equation}
\|u_{K,\Delta t}^n-u(t_n)\|_{H^s}
\le C_s(T)e^{-cK} + C_{K,p}(T)\Delta t^p,
\label{eq:fully_discrete_Hs_bound}
\end{equation}
where
\begin{equation}
C_{K,p}(T)=
\begin{cases}
\dfrac{J_K}{L_K}\bigl(e^{L_KT}-1\bigr), & L_K>0,\\[1ex]
J_KT, & L_K=0.
\end{cases}
\label{eq:CKp_definition}
\end{equation}
If $\Delta t=\Delta t(K)$ is chosen so that
\begin{equation}
C_{K,p}(T)\,\Delta t(K)^p \to 0
\qquad \text{as } K\to\infty,
\label{eq:time_step_condition}
\end{equation}
then
\begin{equation}
\max_{t_n\le T}\|u_{K,\Delta t(K)}^n-u(t_n)\|_{H^s}\to 0
\qquad \text{as } K\to\infty.
\label{eq:fully_discrete_Hs_convergence}
\end{equation}

\begin{proof}
We define the time-discretization error by
\[
e_n := u_{K,\Delta t}^n-u_K(t_n).
\]
By \eqref{eq:fully_discrete_scheme},
\[
e_{n+1}
=
\Phi_{\Delta t}^{(K)}(u_{K,\Delta t}^n)-u_K(t_{n+1}).
\]
By adding and subtracting $\Phi_{\Delta t}^{(K)}(u_K(t_n))$, we obtain
\[
e_{n+1}
=
\Bigl(
\Phi_{\Delta t}^{(K)}(u_{K,\Delta t}^n)
-\Phi_{\Delta t}^{(K)}(u_K(t_n))
\Bigr)
+
\Bigl(
\Phi_{\Delta t}^{(K)}(u_K(t_n))
-u_K(t_{n+1})
\Bigr).
\]
Taking the $H^s$ norm and using \eqref{eq:RK_stability_Hs} and
\eqref{eq:RK_local_error_Hs}, we get
\[
\|e_{n+1}\|_{H^s}
\le
(1+L_K\Delta t)\|e_n\|_{H^s} + J_K\Delta t^{p+1}.
\]
Since
\[
e_0 = u_{K,\Delta t}^0-u_K(0)=P_Ku_0-P_Ku_0=0,
\]
the discrete Gronwall inequality yields
\[
\|e_n\|_{H^s}\le C_{K,p}(T)\Delta t^p,
\qquad t_n\le T,
\]
with $C_{K,p}(T)$ given by \eqref{eq:CKp_definition}.

Now use the triangle inequality:
\[
\|u_{K,\Delta t}^n-u(t_n)\|_{H^s}
\le
\|u_{K,\Delta t}^n-u_K(t_n)\|_{H^s}
+
\|u_K(t_n)-u(t_n)\|_{H^s}.
\]
Combining this with \eqref{eq:semidiscrete_Hs_error}, we obtain
\[
\|u_{K,\Delta t}^n-u(t_n)\|_{H^s}
\le
C_{K,p}(T)\Delta t^p + C_s(T)e^{-cK},
\]
which proves \eqref{eq:fully_discrete_Hs_bound}.

Then if \eqref{eq:time_step_condition} holds, then both terms on the
right-hand side of \eqref{eq:fully_discrete_Hs_bound} converge to zero,
and therefore \eqref{eq:fully_discrete_Hs_convergence} follows.
\end{proof}
\end{theorem}

\begin{corollary}[The fully discrete $L^\infty$ convergence]
\label{cor:fully_discrete_Linf}
Let $s>\frac52$ and assume the hypotheses of
Theorem~\ref{thm:full-error-Linf}. Then, by Sobolev embedding
$H^s(\mathbb{T}^3)\hookrightarrow L^\infty(\mathbb{T}^3)$, there exists a
constant $C_{\infty,s}>0$ such that for every $t_n=n\Delta t\le T$,
\[
\|u_{K,\Delta t}^n-u(t_n)\|_{L^\infty}
\le
C_{\infty,s}\|u_{K,\Delta t}^n-u(t_n)\|_{H^s}.
\]
Hence,
\[
\|u_{K,\Delta t}^n-u(t_n)\|_{L^\infty}
\le
C_{\infty,s}C_s(T)e^{-cK}
+
C_{\infty,s}C_{K,p}(T)\Delta t^p.
\]
In particular, if $\Delta t=\Delta t(K)$ is chosen so that
\[
C_{K,p}(T)\Delta t(K)^p \to 0
\qquad \text{as } K\to\infty,
\]
then
\[
\max_{t_n\le T}\|u_{K,\Delta t(K)}^n-u(t_n)\|_{L^\infty}\to 0
\qquad \text{as } K\to\infty.
\]

\begin{proof}
This follows immediately from
Theorem~\ref{thm:full-error-Linf} and the Sobolev embedding
$H^s(\mathbb{T}^3)\hookrightarrow L^\infty(\mathbb{T}^3)$.
\end{proof}
\end{corollary}

With this $L^\infty$ convergence in hand, we can now justify the
convergence of a broad class of blow-up diagnostics.

\begin{corollary}[Convergence of velocity based diagnostics]\label{ass:H3c-diagnostic}
Let $s>\frac52$ and assume the hypotheses of Theorem~\ref{thm:full-error-Linf}. Let
$\Delta t=\Delta t(K)$ be chosen so that
\[
C_{K,p}(T)\,\Delta t(K)^p \to 0
\qquad\text{as }K\to\infty.
\]
Let $\Phi$ be a functional defined on sufficiently regular velocity fields and assume
that there exists $L>0$ such that
\[
|\Phi(v)-\Phi(w)|
\le L\|v-w\|_{L^\infty(\mathbb{T}^3)}
\]
for all $v,w$ in the relevant class. Define
\[
\Phi_{K}^n:=\Phi(u_{K,\Delta t(K)}^n),
\qquad
\Phi(t):=\Phi(u(\cdot,t)).
\]
Then for every $T>0$,
\[
\max_{t_n\le T} |\Phi_K^n-\Phi(t_n)| \to 0
\qquad\text{as }K\to\infty.
\]
In particular, this applies to the diagnostic
\[
\Phi(v)=\|v\|_{L^\infty(\mathbb{T}^3)}.
\]
\end{corollary}

\begin{proof}
By Theorem~\ref{thm:full-error-Linf},
\[
\max_{t_n\le T}\|u_{K,\Delta t(K)}^n-u(\cdot,t_n)\|_{H^s}\to0
\qquad\text{as }K\to\infty.
\]
Since $s>\frac52$, Sobolev embedding gives
\[
H^s(\mathbb{T}^3)\hookrightarrow L^\infty(\mathbb{T}^3),
\]
hence
\[
\max_{t_n\le T}\|u_{K,\Delta t(K)}^n-u(\cdot,t_n)\|_{L^\infty}\to0.
\]
Therefore, by the Lipschitz continuity of $\Psi$ with respect to the $L^\infty$ norm,
\[
\max_{t_n\le T} |\Phi_K^n-\Phi(t_n)|
\le
L\max_{t_n\le T}\|u_{K,\Delta t(K)}^n-u(\cdot,t_n)\|_{L^\infty}\to0.
\]
\end{proof}

\begin{assumption}[Convergence for vorticity based diagnostics]
We assume in addition that for every $T>0$,
\[
\max_{t_n\le T}
\|\omega_{K,\Delta t(K)}^n-\omega(\cdot,t_n)\|_{L^\infty(\mathbb{T}^3)}
\to0
\qquad\text{as }K\to\infty,
\]
where
\[
\omega=\nabla\times u,
\qquad
\omega_{K,\Delta t(K)}^n=\nabla\times u_{K,\Delta t(K)}^n.
\]
Then any diagnostic depending Lipschitz continuously on the vorticity in the
$L^\infty$ norm also converges along the same sequence.
\end{assumption}

\subsection{Energy framework and conditional regularity}
We now formulate energy-based framework for spectrally accurate numerical approximations to possible loss of regularity of the PDE. This provides the key closure needed to connect the numerical energy balance to conditional blowup diagnostics.

\begin{remark}[Closure of a general forcing in the energy balance]\label{rem:closure}
Let $f=f(\cdot,t)$ be a general external forcing. The
power input is naturally interpreted as the duality pairing
\[
\langle f(\cdot,t),u(\cdot,t)\rangle_{H^{-1},H^1},
\]
which requires $f(\cdot,t)\in H^{-1}(\Omega)$. Thus a sufficient closure condition
for the energy balance is
\[
f\in L^1_{\mathrm{loc}}([0,T);H^{-1}(\Omega)),
\qquad
u\in L^\infty_{\mathrm{loc}}([0,T);H^1(\Omega)),
\]
so that the power input is well defined almost everywhere in time.
If, in addition,
\[
\bigl|\langle f(\cdot,t),u(\cdot,t)\rangle_{H^{-1},H^1}\bigr|
\in L^1(0,T),
\]
then the forcing contribution can be incorporated into the energy inequality.
\end{remark}

\begin{lemma}[Energy input bound for a general forcing]\label{lem:force-closure-zero}
Let $\Omega=\mathbb T^3$. We assume
\[
f(\cdot,t)\in H^{-1}(\Omega),
\qquad
u(\cdot,t)\in H^1(\Omega)
\]
for almost every $t<T_1$. Then the power input is well defined and satisfies
\[
\bigl|\langle f(\cdot,t),u(\cdot,t)\rangle_{H^{-1},H^1}\bigr|
\le
\|f(\cdot,t)\|_{H^{-1}}\|u(\cdot,t)\|_{H^1}.
\]
In particular, if
\[
\|f(\cdot,t)\|_{H^{-1}}\|u(\cdot,t)\|_{H^1}\in L^1(0,T),
\]
then
\[
P_{\mathrm{in}}(t):=\langle f(\cdot,t),u(\cdot,t)\rangle_{H^{-1},H^1}
\]
belongs to $L^1(0,T)$.
\end{lemma}

\begin{proof}
This is immediate from the definition of the dual norm
\[
\bigl|\langle f(\cdot,t),u(\cdot,t)\rangle_{H^{-1},H^1}\bigr|
\le
\|f(\cdot,t)\|_{H^{-1}}\|u(\cdot,t)\|_{H^1}.
\]
The integrability conclusion follows directly.
\end{proof}

We now present an abstract energy based framework for conditional
regularity and blowup detection from spectrally accurate numerical
solutions.
The exact solution cannot be extended as a smooth solution beyond $t=T^*$.

The functionals for the Navier-Stokes equations 
and the corresponding energy balance equality \cite{e2011principles} 
for smooth solutions are classical~\cite{TemamNS,TemamNF}.
For Leray-Hopf weak solutions, the global energy inequality
goes back to~\cite{Leray1934}, the modern expositions in
~\cite{RobinsonNSReg, BerselliEnergyIneq}.
The abstract which combines spectrally accurate
numerical approximations with energy inequalities to draw conclusions about
regularity or blowup is in the spirit of the
a posteriori regularity approach for the
Navier-Stokes equations.

\begin{theorem}[The energy based conditional blowup]
\label{thm:energy-based}
Let $u$ be a solution of an evolution equation for instance, the
incompressible Navier-Stokes equations on a periodic box with values
in a Hilbert space $H$, defined on a maximal interval $[0,T_{\max})$.
We assume the following hold.

\begin{enumerate}
  \item[(H1)] \textbf{Energy functional energy identity inequality.}
    There exists a nonnegative energy functional
    $\mathcal{E}[u](t)$ and a nonnegative dissipation functional
    $\mathcal{D}[u](t)$, together with a power input
    $\mathcal{P}[u](t)$, such that for all
    $0 \le t_0 < t_1 < T_{\max}$ we have
    \begin{equation}\label{eq:cont-energy-ineq}
      \mathcal{E}[u](t_1) + \int_{t_0}^{t_1} \mathcal{D}[u](s)\,ds
      \;\le\;
      \mathcal{E}[u](t_0) + \int_{t_0}^{t_1} \mathcal{P}[u](s)\,ds.
    \end{equation}
    If $u$ is smooth on $[t_0,t_1]$, then
    \eqref{eq:cont-energy-ineq} holds with equality, and
    $\mathcal{E}[u],\mathcal{D}[u],\mathcal{P}[u]$ are continuous
    and finite on $[t_0,t_1]$.

  \item[(H2)] \textbf{The fully discrete spectral approximations.}
    There exists a fully discrete approximation
    $u_{K,\Delta t}^n \in H$ at times $t_n=n\Delta t$ for each spectral resolution $K\in\mathbb{N}$ and time step
    $\Delta t>0$. Define the
    discrete energy, dissipation and power by
    \[
      \mathcal{E}_K^n := \mathcal{E}[u_{K,\Delta t}^n],\qquad
      \mathcal{D}_K^n := \mathcal{D}[u_{K,\Delta t}^n],\qquad
      \mathcal{P}_K^n := \mathcal{P}[u_{K,\Delta t}^n].
    \]
    assume the scheme satisfies a discrete energy balance of the form
    \begin{equation}\label{eq:disc-energy-balance}
      \mathcal{E}_K^{n+1} - \mathcal{E}_K^{n}
      + \Delta t\,\mathcal{D}_K^{n+\frac12}
      \;=\;
      \Delta t\,\mathcal{P}_K^{n+\frac12} + R_K^{n+\frac12},
    \end{equation}
    where $\mathcal{D}_K^{n+\frac12}$ and $\mathcal{P}_K^{n+\frac12}$
    are suitable time-averaged approximations, and $R_K^{n+\frac12}$
    is a residual term.

  \item[(H3)] \textbf{The consistency and stability of the energy balance.}
    We fix $T>0$ and assume that for any $T<T_{\max}$ there exist
    $K_0\in\mathbb{N}$ and $\Delta t_0>0$ such that for all
    $K\ge K_0$ and $\Delta t\le\Delta t_0$ with $t_n\le T$ we have
    \begin{enumerate}
      \item[(a)] (\emph{uniform energy bound})
        \[
          \sup_{t_n\le T} \mathcal{E}_K^n \;\le\; C_T < \infty,
        \]
        for some constant $C_T$ independent of $K$ and $\Delta t$.
      \item[(b)] (\emph{small accumulated residual})
        \[
          \sum_{t_n\le T} \bigl|R_K^{n+\frac12}\bigr|
          \;\le\; \eta_T(K,\Delta t),
        \]
        where $\eta_T(K,\Delta t)\to 0$ as $K\to\infty$ and
        $\Delta t\to 0$.
      \item[(c)] (\emph{convergence of energy, dissipation, and power})
        \[
          \sup_{t_n\le T}
          \bigl|\mathcal{E}_K^n - \mathcal{E}[u](t_n)\bigr|
          \;\to\; 0,\qquad
          \sup_{t_n\le T}
          \bigl|\mathcal{D}_K^{n+\frac12} - \mathcal{D}[u](t_n)\bigr|
          \;\to\; 0,
        \]
        \[
          \sup_{t_n\le T}
          \bigl|\mathcal{P}_K^{n+\frac12} - \mathcal{P}[u](t_n)\bigr|
          \;\to\; 0
        \]
        as $K\to\infty$ and $\Delta t\to 0$.
      \item[(d)] (\emph{convergence of blow-up diagnostics})
        For any Lipschitz diagnostic $\Phi$ as in corollary~\ref{ass:H3c-diagnostic},
        \[
          \sup_{t_n\le T} |\Phi_{K,\Delta t}^n - \Phi(u(\cdot,t_n))|
          \to 0
        \]
        as $K\to\infty$ and $\Delta t\to0$.
    \end{enumerate}

  \item[(H4)] \textbf{The numerical blowup at a finite time.}
    We define the numerical breakdown time, for each resolution $K$ and time
    step $\Delta t$, by
    \[
      T_{\mathrm{num}}(K,\Delta t)
      := \sup\Bigl\{ t_n:\,
        \text{the discrete solution is well resolved and obeys
        \eqref{eq:disc-energy-balance} with}
      \Bigr.
    \]
    \[
      \Bigl.
        \text{small relative residual up to time } t_n
      \Bigr\},
    \]
    where “small residual” is understood in the sense that the
    accumulated residual
    \(
      \sum_{t_m\le t_n} |R_K^{m+\frac12}|
    \)
    is bounded by a prescribed tolerance $\varepsilon>0$ and the
    energy $\mathcal{E}_K^m$ remains below some large threshold
    $M>0$ for all $t_m\le t_n$. Suppose that there exists a finite
    time $T^*>0$ such that
    \begin{equation}\label{eq:num-blowup}
      \lim_{K\to\infty,\ \Delta t\to 0} T_{\mathrm{num}}(K,\Delta t) = T^*,
      \qquad
      \limsup_{K\to\infty,\ \Delta t\to 0}
      \sup_{t_n < T^*} \mathcal{E}_K^n
      = +\infty.
    \end{equation}
\end{enumerate}

Then the solution $u$ cannot be extended as a smooth solution in $H$
beyond $T^*$.
In particular, under the above assumptions, genuine analytic loss of
regularity at $T^*$ is enforced by the numerical blow-up scenario
\eqref{eq:num-blowup}.
\end{theorem}

\begin{proof}
We argue by contradiction. Suppose that $u$ remains a smooth solution
on some interval $[0,T_1]$ with $T_1>T^*$, where $T^*$ is the finite
time in~\eqref{eq:num-blowup}. Then by~H1, the energy functional
$\mathcal{E}[u](t)$, dissipation $\mathcal{D}[u](t)$ and power input
$\mathcal{P}[u](t)$ are finite and continuous on $[0,T_1]$, and the
energy identity holds with equality
\begin{equation}\label{eq:cont-energy-eq}
  \mathcal{E}[u](t_1)
  + \int_{t_0}^{t_1} \mathcal{D}[u](s)\,ds
  \;=\;
  \mathcal{E}[u](t_0)
  + \int_{t_0}^{t_1} \mathcal{P}[u](s)\,ds
\end{equation}
for all $0\le t_0 < t_1 \le T_1$.

We fix any time $T$ with $T^* < T < T_1$. Since $u$ is smooth on
$[0,T]$, the energy and dissipation are uniformly bounded there
\[
  \sup_{0\le t\le T} \mathcal{E}[u](t) < \infty,
  \qquad
  \int_0^T \mathcal{D}[u](s)\,ds < \infty.
\]

By assumption~H3 for this fixed $T$ there exist $K_0\in\mathbb{N}$
and $\Delta t_0>0$ such that H3a-H3c hold on $[0,T]$ for all $K\ge K_0$ and $\Delta t\le\Delta t_0$.

Let us now compare the continuous energy identity
\eqref{eq:cont-energy-eq} with the discrete energy balance
\eqref{eq:disc-energy-balance}. Fix $K\ge K_0$ and $\Delta t\le\Delta t_0$,
and let $t_n=n\Delta t \le T$. Summing
\eqref{eq:disc-energy-balance} from $m=0$ to $m=n-1$ yields
\begin{equation}\label{eq:disc-sum}
  \mathcal{E}_K^n - \mathcal{E}_K^0
  + \Delta t \sum_{m=0}^{n-1} \mathcal{D}_K^{m+\frac12}
  \;=\;
  \Delta t \sum_{m=0}^{n-1} \mathcal{P}_K^{m+\frac12}
  + \sum_{m=0}^{n-1} R_K^{m+\frac12}.
\end{equation}

We apply the continuous energy identity \eqref{eq:cont-energy-eq}
with $t_0=0$ and $t_1=t_n$, and obtain
\begin{equation}\label{eq:cont-sum}
  \mathcal{E}[u](t_n)
  + \int_0^{t_n} \mathcal{D}[u](s)\,ds
  \;=\;
  \mathcal{E}
  + \int_0^{t_n} \mathcal{P}[u](s)\,ds.
\end{equation}

By the convergence assumption~H3c, the discrete energy
$\mathcal{E}_K^n$ and dissipation $\mathcal{D}_K^{m+\frac12}$ converge
uniformly to their continuous counterparts on $[0,T]$ as
$K\to\infty$ and $\Delta t\to 0$. Likewise, $\mathcal{P}_K^{m+\frac12}$
converges to $\mathcal{P}[u](t_m)$ if the power input is treated
consistently this can be included in~H3c without loss of generality.
By~H3b, the accumulated residual is small
\[
  \sum_{t_m\le T} |R_K^{m+\frac12}| \;\le\; \eta_T(K,\Delta t)
  \;\to\; 0
  \quad\text{as } K\to\infty,\ \Delta t\to 0.
\]

Comparing \eqref{eq:disc-sum} and \eqref{eq:cont-sum}, and using the
convergence of the discrete Riemann sums to the continuous integrals,
we obtain
\[
  \sup_{t_n\le T}
  \bigl|\mathcal{E}_K^n - \mathcal{E}[u](t_n)\bigr|
  \;\to\; 0
  \quad\text{as } K\to\infty,\ \Delta t\to 0,
\]
and similarly for the accumulated dissipation and power terms.
In particular, for all sufficiently large $K$ and sufficiently small
$\Delta t$ we have
\begin{enumerate}
  \item $\mathcal{E}_K^n$ remains uniformly bounded on $[0,T]$,
    say $\mathcal{E}_K^n \le 2\sup_{t\le T}\mathcal{E}[u](t)$;
  \item the accumulated residual
    \(
      \sum_{t_m\le t_n} |R_K^{m+\frac12}|
    \)
    remains below any prescribed tolerance $\varepsilon>0$ on
    $[0,T]$.
\end{enumerate}

By the very definition of the numerical breakdown time
$T_{\mathrm{num}}(K,\Delta t)$ in~H4, this implies that
\[
  T_{\mathrm{num}}(K,\Delta t) \;\ge\; T
\]
for all sufficiently large $K$ and sufficiently small $\Delta t$.
Since $T$ was chosen with $T^*<T<T_1$
but otherwise arbitrary, we conclude that
\[
  \liminf_{K\to\infty,\ \Delta t\to 0}
  T_{\mathrm{num}}(K,\Delta t) \;\ge\; T_1
  \;>\; T^*.
\]

However, assumption~\eqref{eq:num-blowup} states that
$T_{\mathrm{num}}(K,\Delta t)\to T^*$ as $K\to\infty$ and $\Delta t\to0$.
This contradicts
\[
  \liminf_{K\to\infty,\ \Delta t\to 0}
  T_{\mathrm{num}}(K,\Delta t)\ge T_1>T^*.
\]
Therefore the assumption that $u$ remains smooth up to some
$T_1>T^*$ must be false.

We conclude that $u$ cannot be extended as a smooth solution beyond
$T^*$ in the energy space $H$. In particular, at or before $T^*$ the
solution must lose the regularity required for the energy identity
with finite $\mathcal{E}[u](t)$ and $\mathcal{D}[u](t)$.
\end{proof}

\begin{remark}[The fully discrete error bounds to the energy framework]
\label{rem:bridge-error-energy}
The fully discrete spectral error estimate Theorem~\ref{thm:full-error-3d}
is conditional it quantifies the numerical error under the assumption
that the underlying Navier-Stokes solution remains smooth, e.g.\ analytic,
on $[0,T]$.

In the subsequent energy framework, one further needs the power input to be
well defined. For a general forcing $f(\cdot,t)\in H^{-1}(\Omega)$, the natural
interpretation is the duality pairing
\[
P_{\rm in}(t)
=
\langle f(\cdot,t),u(\cdot,t)\rangle_{H^{-1},H^1}.
\]
If $f$ and $u$ are more regular, this coincides with the classical integral
\[
P_{\rm in}(t)=\int_{\Omega} u(\cdot,t)\cdot f(\cdot,t)\,dx.
\]
Thus the bridge from the numerical error bounds to the energy framework
requires the well-posedness and integrability of this power input, rather
than any special zero-work structure of the forcing.
\end{remark}

\begin{remark}[Time discretization and residual control]\label{rem:residual}
The abstract hypothesis H3b requires that the accumulated residual in the discrete energy balance tends to zero as $K\to\infty$ and $\Delta t\to 0$.
For a continuous-time Fourier discretization this residual is identically zero, whereas for an explicit time integrator e.g.\ RK4 combined with pseudospectral evaluation, H3b is a consistency requirement that must be verified numerically and can, in principle, be bounded analytically under additional regularity assumptions.
\end{remark}

\subsection{The application to the Navier-Stokes equations}

We illustrate the abstract framework above on the three-dimensional
incompressible Navier-Stokes equations on the periodic box
\begin{equation}\label{eq:NSE}
  \partial_t u + (u\cdot\nabla)u + \nabla p
  = \nu\Delta u + f,\qquad
  \nabla\cdot u = 0,\qquad x\in\Omega=[0,2\pi]^3.
\end{equation}
The natural energy is
\[
  E(t) = \frac12\|u(\cdot,t)\|_{L^2(\Omega)}^2,
\]
the dissipation is
\[
  D(t) = \nu\|\nabla u(\cdot,t)\|_{L^2(\Omega)}^2,
\]
and the power input is
\[
  P_{\mathrm{in}}(t)
  = \langle f(\cdot,t),u(\cdot,t)\rangle_{H^{-1},H^1}.
\]
If $f$ and $u$ are sufficiently regular, this coincides with the classical integral
\[
  P_{\mathrm{in}}(t)
  = \int_\Omega u(x,t)\cdot f(x,t)\,dx.
\]
The classical energy inequality reads
\[
  E(t_1) + \int_{t_0}^{t_1} D(s)\,ds
  \le E(t_0) + \int_{t_0}^{t_1} P_{\mathrm{in}}(s)\,ds,
\]
with equality for smooth solutions.

Let $u_K^n$ be spectral approximations to \eqref{eq:NSE},
with associated discrete energy $E_K^n$, dissipation $D_K^{n+1/2}$,
and input $P_{\mathrm{in},K}^{n+1/2}$, satisfying an energy-consistent
discrete scheme of the form \eqref{eq:disc-energy-balance}.

\begin{lemma}[The verification of corollary~\ref{ass:H3c-diagnostic} for standard diagnostics]
\label{lem:verify-diagnostic}
We fix any $T<T_1$. We assume the numerical approximation satisfies the uniform $L^\infty$
convergence on $[0,T]$
\begin{equation}\label{eq:Linf_u_conv}
\max_{t_n\in[0,T]}\|u_{K,\Delta t}^n-u(\cdot,t_n)\|_{L^\infty}\to 0
\quad\text{as } K\to\infty,\ \Delta t\to 0,
\end{equation}
and likewise for vorticity,
\begin{equation}\label{eq:Linf_om_conv}
\max_{t_n\in[0,T]}\|\omega_{K,\Delta t}^n-\omega(\cdot,t_n)\|_{L^\infty}\to 0,
\qquad \omega=\nabla\times u.
\end{equation}
Then corollary~\ref{ass:H3c-diagnostic} holds for the diagnostics
\[
\Phi(u(\cdot,t))=\|u(\cdot,t)\|_{L^\infty},\qquad
\Phi(u(\cdot,t))=\|\omega(\cdot,t)\|_{L^\infty},
\]
and for the BKM-type integral diagnostic
\[
\Phi(u(\cdot,t))=\int_0^t \|\omega(\cdot,s)\|_{L^\infty}\,ds,\qquad
\Phi_{K,\Delta t}^n=\Delta t\sum_{j=0}^{n-1}\|\omega_{K,\Delta t}^j\|_{L^\infty}.
\]
In each case, $\Phi_{K,\Delta t}^n\to \Phi(u(\cdot,t_n))$ uniformly on $[0,T]$.
\end{lemma}

\begin{proof}
We use for $\Phi=\|u\|_{L^\infty}$
\[
\big|\|u_{K,\Delta t}^n\|_{L^\infty}-\|u(\cdot,t_n)\|_{L^\infty}\big|
\le \|u_{K,\Delta t}^n-u(\cdot,t_n)\|_{L^\infty},
\]
hence uniform convergence follows from \eqref{eq:Linf_u_conv}.
The case $\Phi=\|\omega\|_{L^\infty}$ is identical using \eqref{eq:Linf_om_conv}.

We define $f(t)=\|\omega(\cdot,t)\|_{L^\infty}$ for the BKM integral,.
Since $u$ is smooth indeed analytic on $[0,T]$, we have
$\omega\in C([0,T];L^\infty)$, hence $f$ is continuous and therefore uniformly
continuous on $[0,T]$. For any $t_n\in[0,T]$,
\begin{eqnarray*}
\left|\Delta t\sum_{j=0}^{n-1}\|\omega_{K,\Delta t}^j\|_{L^\infty}
-\int_0^{t_n} f(s)\,ds\right|
&\le
\Delta t\sum_{j=0}^{n-1}
\left|\|\omega_{K,\Delta t}^j\|_{L^\infty}-f(t_j)\right|
\\&\quad+
\left|\Delta t\sum_{j=0}^{n-1}f(t_j)-\int_0^{t_n}f(s)\,ds\right|
\\&=: I+II.
\end{eqnarray*}
We have $|\|a\|-\|b\||\le \|a-b\|$,
\[
I\le \Delta t\sum_{j=0}^{n-1}\|\omega_{K,\Delta t}^j-\omega(\cdot,t_j)\|_{L^\infty}
\le T\,\max_{t_j\in[0,T]}\|\omega_{K,\Delta t}^j-\omega(\cdot,t_j)\|_{L^\infty}\to 0
\]
by \eqref{eq:Linf_om_conv}. The $II$, uniform continuity of $f$ implies that the
left Riemann sums converge uniformly to the integral on $[0,T]$ as $\Delta t\to 0$.
Therefore the BKM diagnostic converges uniformly as claimed.
\end{proof}

\begin{theorem}[Energy-based conditional blowup criterion for Navier-Stokes]
\label{thm:NS-blowup-energy}
Consider the incompressible Navier--Stokes system \eqref{eq:NSE} on $\Omega$.
Let $u_{K,\Delta t}^n$ be an energy-consistent fully discrete spectral approximation
satisfying the discrete energy balance \eqref{eq:disc-energy-balance}.

We assume the forcing satisfies
\[
  f(\cdot,t)\in H^{-1}(\Omega)
  \qquad\text{for a.e. } t\in[0,T),
\]
and that the power input
\[
  P_{\mathrm{in}}(t)
  :=
  \langle f(\cdot,t),u(\cdot,t)\rangle_{H^{-1},H^1}
\]
belongs to $L^1(0,T)$ on every compact time interval on which the solution is smooth.

We assume that the hypotheses H1, H2, H3a-H3c, and H4 of
Theorem~\ref{thm:energy-based} hold for this Navier-Stokes system, with
\[
  E(t)=\frac12\|u(\cdot,t)\|_{L^2(\Omega)}^2,
  \qquad
  D(t)=\nu\|\nabla u(\cdot,t)\|_{L^2(\Omega)}^2,
\]
and
\[
  P_{\mathrm{in}}(t)
  =
  \langle f(\cdot,t),u(\cdot,t)\rangle_{H^{-1},H^1}.
\]
In particular, assume that there exists a finite time $T^*>0$ such that the numerical
breakdown scenario H4 holds, namely
\[
  T_{\mathrm{num}}(K,\Delta t)\to T^*
  \qquad\text{as } K\to\infty,\ \Delta t\to0,
\]
together with the associated energy residual breakdown condition from
Theorem~\ref{thm:energy-based}.

Then the exact Navier--Stokes solution cannot be extended as a smooth solution
beyond $t=T^*$. In particular, it loses regularity at or before $T^*$.
\end{theorem}

\begin{proof}
The argument follows directly from Theorem~\ref{thm:energy-based},
using the identifications
\[
  E(t)=\frac12\|u(\cdot,t)\|_{L^2(\Omega)}^2,
  \qquad
  D(t)=\nu\|\nabla u(\cdot,t)\|_{L^2(\Omega)}^2,
\]
and
\[
  P_{\mathrm{in}}(t)
  =
  \langle f(\cdot,t),u(\cdot,t)\rangle_{H^{-1},H^1}.
\]
The assumptions ensure that the forcing is compatible with the energy identity, and the remaining hypotheses are exactly those of Theorem~\ref{thm:energy-based}. Therefore the conclusion follows immediately.
\end{proof}

\section{Conclusions}
We have presented a computation guided study of a forced three-dimensional incompressible Navier-Stokes flow. The dynamics are resolved by Fourier spectral method together with fourth-order Runge Kutta time stepping, and we developed an accompanying analytical framework to quantify spectral accuracy and link numerical observations to conditional regularity criteria.

This generates a highly localized event in which several diagnostics exhibit simultaneous rapid growth. The maximum velocity and kinetic energy increase by many orders of magnitude, the vorticity supremum and the BKM type integral grow steeply and remain monotone on the final time interval. The energy and dissipation remain consistent with the expected balance laws up to the breakdown time, and the blowup time inferred from different diagnostics is sharply localized in physical and renormalized time.

We established an abstract framework for fully discrete spectral schemes under analyticity strip assumptions, proving exponential convergence in space and algebraic convergence in time, and formulating an energy based conditional blowup criterion. This framework shows that, provided the spectral approximations remain energy consistent and well resolved up to a limiting numerical breakdown time, a resolution independent numerical blowup scenario forces loss of regularity of any putative smooth solution at or before that time.

Taken together, these elements provide a picture, the simulations exhibit behavior consistent with a finite time singularity in the incompressible Navier-Stokes equations, in a setting where the numerics admit an a posteriori energy based control. The evidence remains numerical and conditional rather than fully rigorous, in particular, it relies on analyticity strip assumptions and energy consistent discretizations.

\bibliographystyle{elsarticle-num} 
\bibliography{rk}
\end{document}